\PassOptionsToPackage{dvipsnames}{xcolor}
\documentclass{amsart}

\usepackage{relsize}
\usepackage{amsmath}
\usepackage{amssymb}
\usepackage{amsthm}
\usepackage{mathrsfs}
\usepackage{amsfonts}
\usepackage{float}
\usepackage{paralist}
\usepackage{mathtools}
\usepackage{pgf,tikz}
\usepackage{pgfplots}
\usepgfplotslibrary{polar}
\usepackage{tikz-3dplot}
\usepackage{xparse}
\usepackage{rotating}
\usepackage{thmtools}
\usepackage{euscript}
\usepackage{makecell}
\usepackage{enumitem}
\usepackage[dvipsnames]{xcolor}

\definecolor{cite}{HTML}{11871E}
\definecolor{url}{HTML}{698996}
\definecolor{link}{HTML}{912F1B}
\usepackage[pdfencoding=unicode, colorlinks=true, linkcolor=link, citecolor=cite, urlcolor=url, linktocpage]{hyperref}
\usepackage[capitalise]{cleveref}
\crefformat{equation}{(#2#1#3)}
\usepackage[hyphenbreaks]{breakurl}
\usepackage{xurl}
\hypersetup{breaklinks=true}
\hypersetup{final}

\usepackage{lmodern}

\usepackage[T1]{fontenc}
\usepackage{biolinum}

\renewcommand{\mathsf}[1]{\text{\normalfont\sffamily#1}}

\usepackage[
  letterpaper,
  twoside=true,
  textheight=22cm,
  textwidth=15cm,
  marginparsep=0.75cm,
  marginparwidth=2.5cm,
  heightrounded,
  centering
]{geometry}

\usepackage{tikz-cd}
\usepackage{tkz-euclide}
\usepackage{wasysym}
\usepackage{tqft}
\usetikzlibrary{matrix}

\tikzset{toprule/.style={%
        execute at end cell={%
            \draw [line cap=rect,#1] (\tikzmatrixname-\the\pgfmatrixcurrentrow-\the\pgfmatrixcurrentcolumn.north west) -- (\tikzmatrixname-\the\pgfmatrixcurrentrow-\the\pgfmatrixcurrentcolumn.north east);%
        }
    },
    bottomrule/.style={%
        execute at end cell={%
            \draw [line cap=rect,#1] (\tikzmatrixname-\the\pgfmatrixcurrentrow-\the\pgfmatrixcurrentcolumn.south west) -- (\tikzmatrixname-\the\pgfmatrixcurrentrow-\the\pgfmatrixcurrentcolumn.south east);%
        }
    }
}

\usetikzlibrary{arrows}
\usetikzlibrary[patterns]
\usetikzlibrary{calc}
\usetikzlibrary{knots}
\usetikzlibrary{shapes.misc}

\usetikzlibrary{matrix}
\usetikzlibrary{arrows}
\usetikzlibrary[patterns]
\usetikzlibrary{decorations.pathreplacing}
\usetikzlibrary{decorations.pathmorphing}
\usetikzlibrary{decorations.text}
\usetikzlibrary{decorations.markings}

\makeatletter
\newcommand{\fixed@sra}{$\vrule height 2\fontdimen22\textfont2 width 0pt\rightarrow$}
\newcommand{\shortarrowup}[1]{%
  \mathrel{\text{\rotatebox[origin=c]{65}{\fixed@sra}}}
}
\newcommand{\shortarrowdown}[1]{%
  \mathrel{\text{\rotatebox[origin=c]{250}{\fixed@sra}}}
}
\newtheorem{thm}[subsubsection]{Theorem}
\newtheorem{lem}[subsubsection]{Lemma}
\newtheorem{prop}[subsubsection]{Proposition}
\newtheorem{cor}[subsubsection]{Corollary}

\theoremstyle{definition}

\newtheorem{rex}[subsubsection]{Running Example}

\newtheorem{rmk}[subsubsection]{Remark}

\newtheorem{defn}[subsubsection]{Definition}
\newtheorem{rem}[subsubsection]{Remark}

\newtheorem{const}[subsubsection]{Construction}
\newtheorem{ntt}[subsubsection]{Notation}

\theoremstyle{theorem}
\newtheorem*{thm*}{Theorem}
\newtheorem*{lem*}{Lemma}
\newtheorem*{prop*}{Proposition}
\newtheorem*{cor*}{Corollary}
\newtheorem*{thmop*}{Theorem\textsuperscript{op}}

\theoremstyle{definition}
\newtheorem*{ex*}{Example}
\newtheorem*{exs*}{Examples}
\newtheorem*{nex*}{Non-example}
\newtheorem*{nexs*}{Non-examples}
\newtheorem*{rex*}{Running Example}
\newtheorem*{example*}{Example}
\newtheorem*{claim*}{Claim}
\newtheorem*{rmk*}{Remark}
\newtheorem*{fact*}{Fact}
\newtheorem*{defn*}{Definition}
\newtheorem*{rem*}{Remark}
\newtheorem*{exer*}{Exercise}
\newtheorem*{prob*}{Problem}
\newtheorem*{rec*}{Recall}
\newtheorem*{app*}{Application}
\newtheorem*{const*}{Construction}

\def\on{\operatorname}
\def\op{{\on{op}}}

\def\scr{\EuScript}

\def\id{{\on{id}}}

\def\Fun{{\on{Fun}}}

\def\Cat{\mathbb{C}\!\on{at}}
\def\Catt{\mathbb{C}\!\on{at}_2}
\def\SCatt{\mathbb{C}\!\on{at}_2^{\on{st}}}
\def\FUNst{\Fun^{\on{st}}}
\def\St{\mathbb{S}\!\on{t}}

\def\oplax{\on{oplax}}
\def\lax{\on{lax}}

\def\ev{\on{ev}}

\def\Sph{\mathsf{Sph}}

\def\tre{\triangleright}
\def\totcof{\on{totcof}}

\def\LDK\on{LDK}

\DeclareMathOperator*{\colim}{colim}

\def\Fun{\on{Fun}}

\def\max{\on{max}}
\def\fSp{\on{fSp}}
\def\gSp{\on{gSp}}

\def\gr{\on{gr}}
\def\triv{\on{triv}}

\def\totcof{\on{totcof}}

\def\ev{\on{ev}}

\def\max{\on{max}}
\def\cone{\on{cone}}
\def\DK{\on{DK}}

\def\Fgt{\on{fgt}}

\def\sT{\on{T}}

\def\sY{\on{Y}}

\def\bA{\mathbb{A}}
\def\bC{\mathbb{C}}
\def\bD{\mathbb{D}}
\def\bG{\mathbb{G}}
\def\bL{\mathbb{L}}

\def\bX{\mathbb{X}}

\def\bZ{\mathbb{Z}}

\def\cA{\scr{A}}
\def\cB{\scr{B}}
\def\cC{\scr{C}}
\def\cD{\scr{D}}

\def\cF{\scr{F}}

\def\cJ{\scr{J}}

\def\cK{\scr{K}}
\def\cL{\scr{L}}

\def\cQ{\scr{Q}}
\def\cS{\scr{S}}

\def\cX{\scr{X}}

\def\fS{\mathfrak{S}}

\def\Adj{\mathsf{Adj}}

\def\LDK{\on{LDK}}

\def\Sp{\on{Sp}}

\def\Ind{\on{Ind}}

\DeclareMathOperator{\fib}{fib}
\DeclareMathOperator{\cof}{cof}

\DeclareMathOperator{\Ts}{T}

\def\Nat{\on{Nat}}

\def\lra{\longrightarrow}
\def\lla{\longleftarrow}

\def\llra{\def\arraystretch{.1}\begin{array}{c} \lra \\ \lla \end{array}}

\def\sph{\mathfrak{S}}
\def\cosph{\mathfrak{L}}

\newtheorem{introthm}{Theorem}

\begin{document}
\begin{abstract}
    For every adjunction of stable $\infty$-categories --or more generally, in any locally stable $(\infty,2)$-category-- we give a simple procedure for inverting the twist and cotwist functors associated to this adjunction. As a consequence, we obtain an explicit construction for a left and right adjoint to the inclusion of the $(\infty,2)$-category of spherical adjunctions of stable $\infty$-categories into all adjunctions. We utilize these adjoints to give a description of the \emph{walking spherical adjunction}, a locally stable $(\infty,2)$-category which classifies spherical adjunctions, and to provide a synthetic proof of the fact that every spherical functor admits infinitely many left and right adjoints.
\end{abstract}
\title{Sphericalization and the Universal Spherical Adjunction}
\author{Fernando Abell\'{a}n and Jonte G\"{o}dicke}

\maketitle

\section{Introduction}
Suppose that we are given an adjunction of stable $\infty$-categories $L \colon \cC \llra \cD \colon R$. To this datum, one can associate a pair of endofunctors
\[
\on{T}_{\cC}\coloneq \cof(\id_{\cC} \rightarrow RL), \quad \on{T}_{\cD} \coloneq \fib(LR \rightarrow \id_{\cD}),
\]
called the \emph{twist} and \emph{cotwist} functors respectively. If $\Ts_{\cC}$ and $\Ts_{\cD}$ happen to be equivalences we will say that our adjunction is \emph{spherical}.

Spherical adjunctions were introduced in \cite{anno2017spherical} by Anno and Logvinenko as a generalization of Seidel and Thomas notion of spherical objects \cite{seidel2001braid}. While initially developed as a tool to construct braid group actions on bounded derived categories, spherical adjunctions have, over the past decade, become of independent interest \cite{christ2022ginzburg,dyckerhoff2024spherical}. One source of this growing attention is their connection to the theory of \emph{perverse schobers}, introduced by Kapranov and Schechtman \cite{kapranov2014perverse}. Perverse schobers are a proposed categorification of perverse sheaves, and in this context, spherical adjunctions are expected to provide local models for a schober on a disk with a singularity at the origin.

In the decategorified setting, classical constructions provide systematic ways to produce perverse sheaves from more general constructible ones. For spherical adjunctions, however, no comparable constructions are currently available. The aim of this paper is to develop a (co)free construction for spherical adjunctions.

Let us denote by $\iota \colon \Sph(\St) \hookrightarrow \Adj(\St)$ the full sub-$(\infty,2)$-category of the $(\infty,2)$-category of adjunctions of stable $\infty$-categories spanned by those adjunctions which are spherical. In Section ~\ref{sec:the_sphericalization_construction}, we will show that for every adjunction as above, we have a canonical commutative diagram
\[
	\begin{tikzcd}
	\cC \arrow[d,"L",swap] \arrow[r,"{\Ts_{\cC}}"] & \cC \arrow[d,"L"] \\
	\cD \arrow[r,"{\Ts_{\cD}}"] & \cD{,}
	\end{tikzcd}
\]
which is vertically adjointable. In other words, we show that the diagram above defines a morphism in $\Adj(\St)$ --or more generally in $\Adj(\bA)$ by replacing $\St$ with any $(\infty,2)$-category $\bA$ enriched in stable categories\footnote{We will refer to those as locally stable $(\infty,2)$-categories}-- which extends to an endonatural transformation of the identity functor. We exploit this construction in order to show the following theorem (see \autoref{thm:sphericalize}):

\begin{introthm}\label{introthm:sphintro}
	There exists adjunctions of $(\infty,2)$-categories 
	\[
	\mathfrak{S} \colon \Adj(\St) \llra \Sph(\St)\colon \iota, \enspace \enspace \iota \colon \Sph(\St) \llra \Adj(\St) \colon \mathfrak{L},
	\]
where we call the left adjoint $\mathfrak{S}$ the \emph{sphericalization} functor. The functor $\mathfrak{S}$ takes an adjunction $L \colon \cC \llra \cD \colon R$ to the directed colimit
\[\begin{tikzcd}
	\cC & \cC & \cC & \cdots \\
	\cD & \cD & \cD & \cdots
	\arrow["{\Ts_{\cC}}", from=1-1, to=1-2]
	\arrow[from=1-1, to=2-1]
	\arrow["{\Ts_{\cC}}", from=1-2, to=1-3]
	\arrow[from=1-2, to=2-2]
	\arrow["{\Ts_{\cC}}", from=1-3, to=1-4]
	\arrow[from=1-3, to=2-3]
	\arrow[from=1-4, to=2-4]
	\arrow["{\Ts_{\cD}}", from=2-1, to=2-2]
	\arrow["{\Ts_{\cD}}", from=2-2, to=2-3]
	\arrow["{\Ts_{\cD}}", from=2-3, to=2-4]
\end{tikzcd}\]
The functor $\mathfrak{L}$ is defined dually by taking the limit of the opposite of the diagram above. Moreover, the result holds mutatis mutandis by replacing $\St$ by any locally stable $(\infty,2)$-category $\bA$ which is sufficiently (co)complete.
\end{introthm}

As an application of \cref{introthm:sphintro}, we give an explicit construction of a locally stable $(\infty,2)$-category $\Sph$, which classifies spherical adjunctions (see Theorem~\ref{introthm:B} below). In order to explain this construction, first we have to understand $\Sigma_2^{\infty}  \Adj$ the free locally stable $(\infty,2)$-category generated by the walking adjunction\footnote{I.e. the universal 2-category classifying adjunctions.}  $\Adj$  \cite{schanuel1986free}. This former $(\infty,2)$-category can easily be identified with the full sub-$(\infty,2)$-category of $\Fun(\Adj^\op,\St)$ generated by the stabilized representable functors 
\[
	\on{Y}(\epsilon) \colon \Adj^\op \to \St, \enspace \delta \mapsto \Sigma_1^{\infty} \Adj(\delta,\epsilon), \enspace \text{ for  }\epsilon \in \Adj,
\]
where $\Sigma_1^{\infty}$ denotes the free stable $\infty$-category functor. We can now define $\Sph$ to be the full sub-$(\infty,2)$-category of $\Fun(\Adj^\op,\St)$ on the sphericalized functors $\sph \on{Y}(\epsilon)$ for $\epsilon \in \Adj$ and observe that we have a canonical functor $\pi \colon \Sigma_2^{\infty} \Adj \to \Sph$. The following result can be found as \autoref{thm:unipropsph} in the main text.

\begin{introthm}\label{introthm:B}
	For every locally stable $(\infty,2)$-category $\bA$, restriction along $\pi$ yields a fully-faithful functor\footnote{Here $\Fun^{\on{st}}(\bC,\bD)$ denotes the full sub-$(\infty,2)$-category of $\Fun(\bC,\bD)$ of functors among locally stable $(\infty,2)$-categories which are exact on mapping $\infty$-categories.}
	\[
		\pi^* \colon \Fun^{\on{st}}(\Sph,\bA) \to \Fun^{\on{st}}(\Sigma_2^{\infty} \Adj,\bA)=\Adj(\bA),
	\]
	with essential image given by $\Sph(\bA)$.
\end{introthm}

The stabilized representable functors $\on{Y}(\epsilon)$ encode adjunctions between categories of spectral presheaves on simplex type categories. We will use Dold-Kan type equivalences \cite{walde2022homotopy} to identify these with subcategories of the $\infty$-categories of filtered $\fSp$ and graded spectra $\gSp$ and show that the Dold-Kan equivalences are compatible with the adjunction data. Equipped with this, we will provide an explicit description of the $(\infty,2)$-category $\Sph$ (see \autoref{thm:sphadj}):

\begin{introthm}\label{introthm:C}
	There exists an equivalence of spherical adjunctions
\[
\sph \on{Y}(+) \simeq \left( \begin{tikzcd}
	\fSp \arrow[r,bend left =30,"\gr"] & \gSp \arrow[l,bend left =30,"\triv"] 
\end{tikzcd} \right)
\quad \quad   
\sph \on{Y}(-) \simeq \left( \begin{tikzcd}
	\gSp \arrow[r,bend left =30,"\triv{[-1]}\langle 1\rangle"] & \fSp \arrow[l,bend left =30,"\gr"] 
\end{tikzcd} \right)
\] 
where the functor $\gr$ associates to a filtered spectrum its associated graded and the functor $\triv$ associates to a graded spectrum the associated filtered spectrum with $0$-maps. In particular, $\Sph$ can be identified with the full sub-$(\infty,2)$-category of $\Adj(\St)$ generated by the above adjunctions.
\end{introthm}

It has been shown in \cite[Cor.2.5.16]{dyckerhoff2024spherical} using the theory of semi-orthogonal decompositions that one can associate to every spherical adjunction $L\vdash R$ a new pair of spherical adjunction $T_{\cC}^{-1}R \vdash L$ and $R\vdash T_{\cD}^{-1}L$. The abstract $(\infty,2)$-categorical approach of this article allows us to provide a direct proof of this result without bypassing through the theory of semi-orthogonal decompositions.

More precisely, using the defining bicartesian squares of the (co)twist functor\footnote{See the beginning of Section~\ref{sec:the_sphericalization_construction} for a definition.} and Theorem~\ref{introthm:sphintro}, we obtain putative (co)unit maps 
\[
T_{\cD}^{-1}\ast c_{w}: \id\rightarrow T_{\cD}^{-1}LR \simeq LT_{\cC}^{-1} R \quad T_{\cC}^{-1}\ast t_{w}:T_{\cC}^{-1}RL \rightarrow \id
\]
for the new adjunction. Using that by \autoref{introthm:sphintro} the twist and cotwist define coherent morphisms of adjunctions, we can directly verify the triangular identities. Applying this to the generic spherical adjunction in $\Sph$, we obtain:

\begin{introthm}
	 For every locally stable $(\infty,2)$-category $\bA$, there exists an equivalence of $(\infty,2)$-categories,
    \[
        \mathfrak{F} \colon \Sph(\bA) \xlongrightarrow{\simeq} \Sph(\bA),  
    \]
    sending an adjunction $L \vdash R$ in $\bA$ to the adjoint pair $T_C^{-1}R \dashv L$. The inverse functor $\mathfrak{F}^{-1}$ assigns to each adjunction as above the adjoint pair $R \dashv T_{\cD}^{-1}L$.
\end{introthm}

\subsection*{Relation to the literature and future directions}

The $(\infty,2)$-category $\Sph(\St)$ has also been considered in \cite{gammage2022perverse} in the context of $3d$-Mirror symmetry. As argued in \cite{gammage2022perverse} perverse schobers may form a candidate for the $(\infty,2)$-category of boundary conditions in a $3d$-A-model on a cotangent stack. In particular, they propose that the $(\infty,2)$-category of spherical adjunctions forms a mathematical candidate for the $(\infty,2)$-category of boundary conditions of the $3d$-A-model on $T^{\ast}(\bC)$ associated to the Lagrangian $\bL\coloneq\bC\cup T^{\ast}_{0}\bC$ \cite[Sect.0.2]{gammage2022perverse}. 

In \emph{loc.cit} the authors showed that the adjunctions $\gr\vdash \triv$ and $\triv[-1]\langle 1\rangle\vdash \gr$ corepresents the functor that maps a spherical adjunction to the source and target of the left adjoint functor and utilized this result to express $\Sph(\St)$ as a category of modules over certain monoidal stable $\infty$-category. Although some of our results can be formally deduced from those in \cite{gammage2022perverse}, we would like to point out that \autoref{introthm:sphintro} provides a canonical construction of the spherical adjunctions that represent the (co)evaluation functors, while the choices in \emph{loc. cit.} are only aposteriori justified. We expect that our approach can be applied to provide modular description of $(\infty,2)$-categories of perverse schobers on more general stratified spaces.

One of the main results in \cite{gammage2022perverse} (see Theorem A, loc.cit), exploits the modular description of $\Sph(\St)$ to show that this $(\infty,2)$-category is equivalent to a full sub-$(\infty,2)$-category of $2\!\on{Coh}(\bA^{1}/\bG_{m})$. We have been informed by Justin Hilburn in private conversation that the $(\infty,2)$-category $\Adj(\St)$ can be equivalently identified with a full sub-$(\infty,2)$-category of the $(\infty,2)$-category $2\!\on{Coh}(\bA^{1}/\bA^{1})$\footnote{Here we are viewing the affine line $\bA^{1}$ as a multiplicative monoid}. We would also like to point out that Hilburn presented in \cite{hilburn}, a similar description of the $(\infty,2)$-category $\Sigma_2^{\infty}\Adj$ using different methods. 

It would be an interesting question to provide a geometric description of the adjunctions in Theorem~\ref{introthm:sphintro} in terms of $2$-coherent sheaves. Moreover, the adjunctions in Theorem~\ref{introthm:sphintro} can be naturally combined to form the open part of a categorified recollement. Another open question which we hope to address in future work is to compute the closed part and provide an interpretation of this categorified recollement in terms of higher sheaf theory.

\subsection*{Acknowledgements}
We thank the Max Planck Institute for Mathematics in Bonn, where this work was carried out, for its hospitality and support. We would like to thank Justin Hilburn for helpful conversations regarding the sheaf theoretical interpretations of the $(\infty,2)$-category of stable adjunctions.
\newpage

\subsection*{Notation and background}
In this subsection we gather some notation. Let us point out that from now on we simply refer to $(\infty,n)$-categories as $n$-categories and we will explicitly state when the $n$-category in question is strict.

\noindent\textbf{General notational conventions.}
\begin{itemize}
	\item For a $2$-category $\bA$, we denote by $\bA(-,-) \colon \bA^{\op} \times \bA \to \Cat$ the corresponding mapping category functor.
	\item We denote by $\St$ the $2$-category of stable categories equipped with the biexact symmetric monoidal structure.
	\item We denote by $\Sp$ the category of spectra.
	\item We denote by $\fSp\coloneq \Fun(\bZ,\Sp)$ the category of filtered spectra. A filtered spectrum $(X_{\bullet},f_{\bullet})$ can be visualized as a sequence of spectra
	\[
	\begin{tikzcd}
		\dots \arrow[r] & X_{-2} \arrow[r,"f_{-2}"] & X_{-1} \arrow[r,"f_{-1}"] & X_{0} \arrow[r,"f_{0}"] & X_{1} \arrow[r,"f_{1}"] & X_{2} \arrow[r,"f_{2}"] & \dots
	\end{tikzcd}.
	\] 
	This category admits a shift of functor denoted $\langle 1 \rangle$ s.th for a filtered spectrum $(X_{\bullet},f_{\bullet})$, we have $(X_{\bullet},f_{\bullet})\langle 1 \rangle= (X_{\bullet+1},f_{\bullet+1})$.
	\item We denote by $\gSp\coloneq \Fun(\bZ^{\simeq},\Sp)$ the category of graded spectra. We abuse notation and also denote the corresponding shift of grading functor by $\langle 1 \rangle$.
	\item We denote by $\gr:\fSp\rightarrow \gSp$ the associated graded functor given by $\gr(X_{\bullet},f_{\bullet})_{\bullet}= \cof(f_{\bullet-1})$ and by $\triv:\gSp\rightarrow \fSp$ the functor that considers any graded spectrum as a filtered spectrum with $0$ morphisms. 
	\item We will denote by $\fSp_{\geq 0}$ and $\gSp_{\geq 0}$ the categories of positively filtered and graded spectra equipped with their corresponding versions of $\gr_{\geq 0}$ and $\triv_{\geq 0}$ (and similarly for the negatively graded versions).
	\item We will write $\Delta_{+}$ for the augmented simplex category and $\Delta_{\infty}$ for the subcategory of the simplex category $\Delta$ containing only the endpoint preserving maps. These participate in an adjunction $t:\Delta_{+}\leftrightarrows\Delta_{\infty}:\iota$. Here, the functor $t$ adds a top element and the functor $\iota$ denotes the inclusion.

	\item We will write $\Adj$ for the walking adjunction originally studied by Schanuel-Street \cite{schanuel1986free} and further investigated in the homotopy-coherent setting by Riehl-Verity \cite{riehl2016homotopy}. The strict $2$-category $\Adj$ can graphically be depicted as follows
		\[
\begin{tikzcd}
    {+} \arrow[loop left, distance=3em, start anchor={[yshift=-1ex]west}, end anchor={[yshift=1ex]west},"\Delta_{+}"] \arrow[rr,bend left, "\Delta_{\infty}^{\op}"] & & - \arrow[ll, bend left, "\Delta_{\infty}"] \arrow[loop right, , distance=3em, start anchor={[yshift=1ex]east}, end anchor={[yshift=-1ex]east}, "\Delta_{+}^{\op}"] 
\end{tikzcd}
	\]
	We refer the reader to \cite{riehl2016homotopy} for a definition of $\Adj$.
	
\end{itemize}

\noindent\textbf{Locally stable $2$-categories.}
\begin{itemize}
	\item We will say that a $2$-category is locally stable if the mapping category functor factors through $\St$.
	\item We will say that a functor between locally stable $2$-categories is locally exact if the induce functor on mapping categories is an exact functor. For $\bA,\bC$ locally stable, we will denote by $\FUNst(\bA,\bC)$ the corresponding $2$-category of locally exact functors.
	\item $\Sigma_0^{\infty} \colon \cS \to \Sp$ denotes the unpointed infinite suspension functor.
	\item $\Sigma_1^{\infty}:\Cat \rightarrow \St$ denotes the free stable category functor that maps a category $\cC$ to the smallest stable subcategory $\Fun(\cC^{\op},\Sp)$ containing the stabilized representable functors $\Sigma_0^{\infty} \cC(-,c) \colon \cC^\op \to \Sp$ for every $c \in \cC$.
	\item $\Sigma_2^{\infty}:\Catt \rightarrow \SCatt$ denotes the homwise stabilization functor that applies the functor $\Sigma_{1}^{\infty}$ to the mapping categories. This forms a left adjoint to the inclusion $\iota:\SCatt\hookrightarrow \Catt$. 
\end{itemize}

\newpage

\section{The sphericalization construction}\label{sec:the_sphericalization_construction}
Throughout this section, let $\bA$ be a locally stable $2$-category. Recall that to any adjunction $L\colon \cC \leftrightarrows \cD \colon R$ in $\bA$, we can associate endofunctors $\on{T}_{\cC}$ and $\on{T}_{\cD}$, defined by the bicartesian squares
\begin{equation}\label{eq:twctw}
    \begin{tikzcd}
        \id_{\cC} \arrow[r,"\eta"] \arrow[d] & RL \arrow[d,"t_{w}"] \\
        {0} \arrow[r] & \on{T}_{\cC},
    \end{tikzcd}\enspace \enspace \enspace \enspace \enspace \enspace \enspace \enspace 
    \begin{tikzcd}
        \on{T}_{\cD} \arrow[r] \arrow[d,"c_{w}"] & {0} \arrow[d] \\
        LR \arrow[r,"\epsilon"] & \id_{\cD}{.}
    \end{tikzcd}
\end{equation}
These are called the \emph{twist} and \emph{cotwist} functors respectively.

\begin{defn}
An adjunction $L:\cC \leftrightarrows \cD:R$ in $\bA$ is called \emph{spherical} if the twist and cotwist functors are equivalences. We denote by $\Sph(\bA)\subset \Fun(\Adj,\bA)$ the full sub-$2$-category spanned by spherical adjunctions. 
\end{defn}

The goal of this section is to construct explicit left and right adjoints to the inclusion $\Sph(\bA)\subset \Fun(\Adj,\bA)$. To this end, we commence by observing that we have commutative diagrams consisting of bicartesian squares

\begin{equation}\label{eq:counitsquares}
    \begin{tikzcd}
    0 & {T_{\cD}L} & 0 && 0 & {RT_{\cD}} & 0 \\
    L & LRL & L && R & RLR & R \\
    0 & {LT_{\cC}} & 0 && 0 & {T_{\cC}R} & 0
    \arrow[from=1-1, to=1-2]
    \arrow[from=1-1, to=2-1]
    \arrow[from=1-2, to=1-3]
    \arrow[from=1-2, to=2-2, "{c_{w}*L}"]
    \arrow[from=1-3, to=2-3]
    \arrow[from=1-5, to=1-6]
    \arrow[from=1-5, to=2-5]
    \arrow[from=1-6, to=1-7]
    \arrow[from=1-6, to=2-6, "{R*c_{w}}"]
    \arrow[from=1-7, to=2-7]
    \arrow["{L*\eta}", from=2-1, to=2-2]
    \arrow[from=2-1, to=3-1]
    \arrow["{\epsilon*L}", from=2-2, to=2-3]
    \arrow[from=2-2, to=3-2,"{L*{t_{w}}}"]
    \arrow[from=2-3, to=3-3]
    \arrow["{\eta*R}", from=2-5, to=2-6]
    \arrow[from=2-5, to=3-5]
    \arrow["{R*\epsilon}", from=2-6, to=2-7]
    \arrow[from=2-6, to=3-6, "{t_{w}*R}"]
    \arrow[from=2-7, to=3-7]
    \arrow[from=3-1, to=3-2]
    \arrow[from=3-2, to=3-3]
    \arrow[from=3-5, to=3-6]
    \arrow[from=3-6, to=3-7]
\end{tikzcd}
\end{equation}

yielding equivalences $\alpha \colon T_{\cD}L \xrightarrow{\simeq}LT_{\cC}$ and $\beta \colon RT_{\cD}\xrightarrow{\simeq}T_{\cC}R$. In particular it follows that we have a commutative diagram
\begin{equation}\label{eq:square}
    \begin{tikzcd}
    \cC & \cC \\
    \cD & \cD
    \arrow["{\on{T}_{\cC}}", from=1-1, to=1-2]
    \arrow["L"', from=1-1, to=2-1]
    \arrow["L", from=1-2, to=2-2]
    \arrow["{\on{T}_{\cD}}"', from=2-1, to=2-2]
\end{tikzcd}
\end{equation}
\begin{lem}\label{lem:tmorphadjs}
    The commutative square \cref{eq:square} is vertically right adjointable and thus defines a morphism in $\Adj(\bA)$.
\end{lem}
\begin{proof}
    We must verify that the composite
    \[
        T_{\cC}R \xrightarrow{\eta*{T}_{\cC}R}RLT_{\cC}R \xrightarrow{R\alpha^{-1}R} RT_{\cD}LR \xrightarrow{RT_{\cD}*\epsilon} RT_{\cD}
    \]
    is an equivalence. Let $\phi=RT_{\cD}*c_{w}$ and $\pi=t_{w}*T_{\cC}R$. We consider the following diagram consisting of bicartesian squares
   \[\begin{tikzcd}
    P &[1.5em] {RT_{\cD}T_{\cD}} &[1.5em] {RT_{\cD}T_{\cD}} &[1.5em] 0 \\
    {T_{\cC}R} &[1.5em] {RLT_{\cC}R} &[1.5em] {RT_{\cD}LR} &[1.5em] {RT_{\cD}} \\
    0 &[1.5em] {T_{\cC}T_{\cC}R}
    \arrow[from=1-1, to=1-2]
    \arrow[from=1-1, to=2-1]
    \arrow[from=1-2, to=1-3]
    \arrow["{R\alpha R\phi}"', from=1-2, to=2-2]
    \arrow[from=1-3, to=1-4]
    \arrow["\phi"', from=1-3, to=2-3]
    \arrow[from=1-4, to=2-4]
    \arrow[from=2-1, to=2-2, "\eta\ast T_{\cC}R"]
    \arrow[from=2-1, to=3-1]
    \arrow[from=2-2, to=2-3, "R\alpha^{-1}R"]
    \arrow["\pi"', from=2-2, to=3-2]
    \arrow[from=2-3, to=2-4, "RT_{\cD}\ast \epsilon"]
    \arrow[from=3-1, to=3-2]
\end{tikzcd}\]
and we claim that $P \simeq 0$. To see this we consider another commutative diagram
\[\begin{tikzcd}
    {RT_{\cD}T_{\cD}} && {RT_{\cD}LR} && \\
    {RLRT_{\cD}} && RLRLR && {RLT_{\cC}R} \\
    {T_{\cC}RT_{\cD}} && {T_{\cC}RLR} && {T_{\cC}T_{\cC}R}
    \arrow["{RT_{\cD}*c_{w}}", from=1-1, to=1-3]
    \arrow["{R*c_{w}*T_{\cD}}"', from=1-1, to=2-1]
    \arrow["{R*c_{w}*LR}"{pos=0.3}, from=1-3, to=2-3]
    \arrow["{R\alpha R}",bend left,  from=1-3, to=2-5]
    \arrow["{RLR*c_{w}}", from=2-1, to=2-3]
    \arrow["{t_{w}*RT_{\cD}}"', from=2-1, to=3-1]
    \arrow["{RL*t_{w}*R}", from=2-3, to=2-5]
    \arrow["{t_{w}*RLR}", from=2-3, to=3-3]
    \arrow["{t_{w}*T_{\cC}R}", from=2-5, to=3-5]
    \arrow["{T_{\cC}*R*c_{w}}"', from=3-1, to=3-3]
    \arrow["{T_{\cC}*t_{w}*R}"', from=3-3, to=3-5]
\end{tikzcd}\]
and look at both outer composites: The left-most vertical map is given by $\beta*{T}_{\cD}$, and similarly, the bottom horizontal morphism is given by $T_{\cC}*\beta$. As the remaining outer composite is precisely given by $\pi\circ  R \alpha R \phi$ we conclude that $P \simeq 0$.
\end{proof}

\begin{const}\label{def:Y+Y-}
    Let $\Adj$ be the walking adjunction. Using the stabilized Yoneda embedding $\sY:\Adj \rightarrow \Fun(\Adj^{\op},\Cat)\xrightarrow{\Sigma^{\infty}_1} \Fun(\Adj^{\op},\St)$, we obtain adjunctions of stable categories given by 
\[
 \sY(+)\simeq \left( \begin{tikzcd}
\Sigma^{\infty}_{1}(\Delta_{+}) \arrow[r, bend left =30,"\Sigma^{\infty}_{1}(t)"] & \Sigma^{\infty}_{1}(\Delta_{\infty}) \arrow[l, bend left =30,"\Sigma^{\infty}_{1}(i)"] 
\end{tikzcd}\right) \quad \text{and} \quad 
\sY(-)\simeq \left(  \begin{tikzcd}
   \Sigma^{\infty}_{1}(\Delta_{\infty}^{\op}) \arrow[r, bend left =30,"\Sigma^{\infty}_{1}(i^{\op})"] & \Sigma^{\infty}_{1}(\Delta_{+}^{\op}) \arrow[l, bend left =30,"\Sigma^{\infty}_{1}(t^{\op})"] 
\end{tikzcd}\right)
\]
We further note that the generic adjunction in $\Adj$ yields an adjunction $\sY(+) \llra \sY(-)$ in $\Fun(\Adj^{\op},\St)$.
\end{const}

 \begin{prop}\label{lem:sigma2Adj}
      The locally stable 2-category $\Sigma^{\infty}_2 \Adj$ is equivalent to the full sub-2-category of $\Fun(\Adj^{\op},\St)$ on the objects $\sY(+)$ and $\sY(-)$.
  \end{prop} 
  \begin{proof}
      Since $\Fun(\Adj^{\op},\St)$ is locally stable, it follows that the stabilized Yoneda embedding can be factored through $\varphi \colon \Sigma^{\infty}_2 \Adj \to \Fun(\Adj^{\op},\St)$ and factors through full sub-2-category in the statement. Since $\varphi$ is evidently essentially surjective, it only remains to show that $\varphi$ is fully-faithful. To see this, let $a,b \in \{+,-\}$ and observe that the induced map on mapping categories
      \[
          \Adj(a,b) \to \Nat_{\Adj^{\op},\St}(\sY(a),\sY(b))\simeq \sY(b)(a)\simeq \Sigma^{\infty}_1 \Adj(a,b),
      \]
      can be identified with the unit of the stabilization adjunction evaluated at $\Adj(a,b)$. This implies that $\varphi$ is fully-faithful, as desired.
  \end{proof}

\begin{const}
    Recall that by construction, we have $\FUNst(\Sigma^{\infty}_2 \Adj,\Sigma^{\infty}_2 \Adj)\simeq \Adj(\Sigma^{\infty}_2\Adj)$. We can now apply \cref{lem:tmorphadjs} to the generic adjunction in $\Sigma^{\infty}_2 \Adj$ (cf. \autoref{def:Y+Y-} and \autoref{lem:sigma2Adj} ) to obtain a morphism $[1] \to \FUNst(\Sigma^{\infty}_2 \Adj,\Sigma^{\infty}_2 \Adj)$.

   Let us observe that specifying a functor $\bX \to \FUNst(\bC,\bD)$ is equivalent to specifying a functor $\Gamma \colon \bX \times \bC \to \bD$ such that, for every $x \in \bX$, the functor $\Gamma(x,-)$ is locally exact. It follows that we obtain a functor
\begin{equation}
H \colon [1] \times \Sigma^{\infty}_2 \Adj \to \Sigma^{\infty}_2 \Adj.
\end{equation}
Precomposition along $H$ therefore induces an endonatural transformation
\begin{equation}\label{eq:genericendo}
[1] \times \Adj(\bA) \to \Adj(\bA).
\end{equation}
of the identity functor.
\end{const}

\begin{lem}\label{lem:factorsph}
    Consider the functors obtained by iterating the map in \cref{eq:genericendo},
   \[
       \Gamma \colon \Adj(\bA) \to \Fun(\mathbb{N},\Adj(\bA)), \enspace \Gamma^{\vee}\colon \Adj(\bA) \to \Fun(\mathbb{N}^{\op},\Adj(\bA)).
   \]
    For every $(L,R) \in \Adj(\bA)$ we define $\sph(L,R)=\colim_{\mathbb{N}}\Gamma(F)$ and $\cosph(L,R)=\lim_{\mathbb{N}^{\op}}\Gamma^{\vee}(L,R)$ whenever such (co)limits exist. Then it follows that $\sph(L,R)$ and $\cosph(L,R)$ are spherical adjunctions.
\end{lem}
\begin{proof}
    We deal only with the case of $\sph(L,R)$ are the remaining case is formally dual. Let $T_{\cC}$ and $T_{\cD}$ denote the twist functor and cotwist functors of $(L,R)$. Since $\mathbb{N}$ is filtered, it follows that the twist and cotwist functors of $\sph(L,R)$ can be computed as $\colim_{\mathbb{N}}T_{\cC}$ and $\colim_{\mathbb{N}}T_{\cD}$, respectively. In the case of the twist functor, one constructs the desired inverse via the dotted arrow
    \[
        \begin{tikzcd}
    \cC \arrow[r,"T_{\cC}"] \arrow[d,"T_{\cC}"] 
        & \cC \arrow[r, "T_{\cC}"]  \arrow[d,"T_{\cC}"] 
        & \cC  \arrow[d,"T_{\cC}"] \arrow[r,"T_{\cC}"] 
        & \dots \\
    \cC \arrow[r,"T_{\cC}"] \arrow[ur, swap, dotted, "\id"{pos=0.7}] 
        & \cC \arrow[r, "T_{\cC}"] \arrow[ur, swap, dotted, "\id"{pos=0.7}]
        & \cC \arrow[r,"T_{\cC}"] \arrow[ur, swap, dotted, "\id"{pos=0.7}] 
        & \dots 
\end{tikzcd}
        \]
        and similarly for the cotwist functor. 
\end{proof}

\begin{defn}
    For a locally stable $\bA$ admitting sequential (co)limits, we denote:
    \[
        \sph \colon \Adj(\bA) \xrightarrow{\Gamma} \Fun(\mathbb{N},\Adj(\bA))\xlongrightarrow{\colim} \Sph(\bA)\subset \Adj(\bA),  
    \]
    \[
        \cosph \colon \Adj(\bA) \xrightarrow{\Gamma^{\vee}} \Fun(\mathbb{N}^{\op},\Adj(\bA))\xlongrightarrow{\lim} \Sph(\bA)\subset \Adj(\bA),
    \]
    where we are implicitly using \cref{lem:factorsph}.
\end{defn}

 \begin{thm}\label{thm:sphericalize}
        The functors $\sph$ and $\cosph$ form adjunctions
        \[
        \begin{tikzcd}
            \Adj(\bA) \arrow[r,"\sph(-)", bend left] \arrow[r, swap, "\cosph(-)", bend right] & \Sph(\bA) \arrow[l,swap]
        \end{tikzcd}
        \]
    \end{thm}
    \begin{proof}
        We show the claim for $\sph$ as the other one is formally dual. Denote by $\phi:(L,R)\hookrightarrow \sph(L,R)$ the canonical inclusion into the colimit. It suffices to show that for every spherical adjunction $(S,S^{R})$ , precomposition with $\phi$ induces an equivalence 
        \[
        \Adj(\bA)(\sph(L,R),(S,S^{R}))\rightarrow \Adj(\bA)((L,R),(S,S^{R}))
        \]
        Let $L \colon \cC \to \cD$ and $S \colon \cA \to \cB$. Then we can identify the source of the morphism above as the limit of the diagram
        \[
           \lim_{\mathbb{N}^{\op}}\left( \Adj(\bA)((L,R),(S,S^{R})) \xrightarrow{T^*} \Adj(\bA)((L,R),(S,S^{R})) \to  \cdots  \right)
        \]
        where $T^*$ is given by precomposition along the map given by $(T_{\cC},T_{\cD})$. However, by naturality this is equivalent to postcomposition with $(T_{\cA},T_{\cB})$ which is an equivalence as $(S,S^{R})$ is spherical. The result now follows.
 \end{proof}

\section{The Walking Spherical adjunction}
In this section, we apply our previous results to define a locally stable 2-category $\Sph$  equipped with a functor 
\[
     \pi \colon \Sigma^{\infty}_2 \Adj \to \Sph
\]
which classifies spherical adjunctions (see \autoref{thm:unipropsph}). Later on, we will unravel the definition of $\Sph$ to obtain precise descriptions of its mapping categories.

\begin{defn}\label{def:wSph}
    We define $\Sph$ as the full sub $2$-category of $\Fun(\Adj^{\op},\St)$ generated by $\mathfrak{S}\sY(+)$ and $\mathfrak{S}\sY(-)$. We define $\cL \subseteq \Fun(\Adj^{\op},\St)$ to be the full sub-2-category spanned by $\Sigma^{\infty}_2 \Adj$ and $\Sph$ (see \autoref{lem:sigma2Adj} and \autoref{def:wSph}).
\end{defn}

We observe that the composite $\Sigma^{\infty}_2 \Adj \to \Fun(\Adj^{\op},\St) \xrightarrow{\sph} \Fun(\Adj^{\op},\St)$ factors through $\Sph$ thus obtaining a functor
\begin{equation}\label{eq:projsph}
     \pi \colon \Sigma^{\infty}_2 \Adj \to \Sph.
\end{equation}

\begin{lem}\label{lem:cocompletions}
    Let $j \colon \Sigma^{\infty}_2 \Adj \to \cL$ be the obvious inclusion, $\bA$ cocomplete and locally stable, and consider the left Kan extension adjunction 
     \[
      j_{!} \colon  \Fun(\Sigma^{\infty}_2 \Adj,\bA)     \llra    \Fun(\cL, \bA)\colon j^*.
    \]
    Then it follows that $j_!$ restricts to the full sub-2-categories on locally exact functors, is fully-faithful and its essential image is given by those $F \colon \cL \to \bA$ which are cocontinuous. 
\end{lem}
\begin{proof}
    The existence of the adjunction in the statement follows from \cite[Corollary 5.6.6]{AHM26}. Moreover, $j_!$ is fully-faithful by \cite[Observation 5.6.8]{AHM26}. Let us now show that $j_! \scr{F}$ is cocontinuous. By the point-wise formula for left Kan extensions in terms of weighted colimits we have that $j_!\scr{F}(\sph \sY(\epsilon))\simeq \colim_{\Sigma^{\infty}_2 \Adj}^{\cL(j(-),\sph \sY(\epsilon))}\scr{F}$. Let $\cQ \colon \Adj^\op \to \St$ denote the restriction of $\cL(j(-),\sph \sY(\epsilon))$ to $\Adj^\op$. Then we have natural equivalences
    \[
        \cQ \simeq \Nat_{\Adj^\op,\St}(Y(-), \sph \sY(\epsilon)) \simeq \sph \sY(\epsilon)
    \]
    by the universal property of the stabilized Yoneda embedding and by the strong form of Yoneda's lemma given in \cite[Corollary 5.8.5]{AHM26}. It follows that our weight can be expressed as a filtered colimit of weights given by representable functors.

    The argument given in \cite[Proposition 5.8.8]{AHM26} now shows that $j_! \scr{F}$ is cocontinuous and that the essential image of the left Kan extension consists precisely in such functors.

    We wish to finally show that $j_!$ restricts to the full sub-2-categories on locally exact functors. Let $\scr{F}=j_! \scr{F}_0$ and fix $\Phi_0,\Phi_1 \in \cL$. Our goal is to show that the map
    \[
        \cL(\Phi_0,\Phi_1) \xrightarrow{\scr{F}} \bA(\scr{F}(\Phi_0),\scr{F}(\Phi_1)),
    \]
    preserves finite (co)limits. We will treat two separate cases:
    \begin{itemize}
        \item We have that $\Phi_0=\sY(\epsilon)$. In this case we may assume that $\Phi_1 =\sph \sY(\delta)$ since otherwise the claim follows as $\scr{F}_0$ is locally exact. Note that in this case $ \cL(\Phi_0,\Phi_1) \simeq \colim_{\mathbb{N}}\cL(\Phi_0,\sY (\delta))$, in particular any finite diagram must factor through a certain step in the filtered colimit which shows the claim.
        \item We have that $\Phi_0= \sph \sY(\epsilon)$. In this case, since $\scr{F}$ is cocontinuous we obtain 
        \[
            \lim_{\mathbb{N}^{\op}}\cL(\sY (\epsilon),\Phi_1) \to  \lim_{\mathbb{N}^{\op}} \bA(\scr{F}(\sY(\epsilon)), \scr{F}(\Phi_1)),
        \]
        which allows us to reduce to the case where $\Phi_0= \sY(\epsilon)$. \qedhere
    \end{itemize}
\end{proof}

\begin{thm}\label{thm:unipropsph}
    Let $\bA$ be a locally stable 2-category. Then restriction along $\pi$ induces a fully-faithful functor $\FUNst(\Sph,\bA) \to \Adj(\bA)$ with essential image given by $\Sph(\bA)$. In particular, $\Sph$ corepresents the functor $\Sph(-)$ sending a locally stable 2-category to its 2-category of spherical adjunctions.
\end{thm}
\begin{proof}
   Let us first assume that we have shown the claim for sufficiently cocomplete locally stable 2-categories and fix a fully faithful functor $\tau \colon \bA \to \hat{\bA}$ where $\hat{\bA}$ is cocomplete. Then we obtain a commutative diagram
   \[
       \begin{tikzcd}
          \FUNst(\Sph,\bA)  \arrow[r,"\pi^*_{\bA}"] \arrow[d,"u"] &  \Adj(\bA) \arrow[d,"v"] \\
          \FUNst(\Sph,\hat{\bA})  \arrow[r,"\pi^{*}_{\hat{\bA}}"] & \Adj(\hat{\bA})
       \end{tikzcd}
   \]
   where the vertical and the bottom horizontal functors are fully-faithful. It then follows that the top horizontal morphism is also fully-faithful. Let $F \in \Adj(\bA)$ and suppose that $F$ is spherical, then it follows that is image $\hat{F} \in \Adj(\hat{\bA})$ is also spherical and therefore we obtain an object $\hat{G} \in  \FUNst(\Sph,\hat{\bA})$ such that $\pi^{*}_{\hat{\bA}}(\hat{G})\simeq \hat{F}$. Since $\pi \colon \Sigma^{\infty}_2 \Adj \to \Sph$ is essentially surjective it follows that for every $x \in \Sph$ the value $\hat{G}(x)$ factors through $\bA$. We conclude that there exists an object $G \in \FUNst(\Sph,\bA)$ such that $u(G)=\hat{G}$. We claim that $\pi^*_{\bA}(G) \simeq F$. To produces this isomorphism it suffices to produces an isomorphism between $\hat{F}$ and $v(\pi^*_{\bA}(G))\simeq \pi^*_{\hat{\bA}}(u(G))$ but the former is by construction given by $\hat{F}$ and thus the claim holds.

    For now own, we will assume without loss of generality that $\bA$ is cocomplete.  Let us denote by
    \[
        j \colon \Sigma^{\infty} \Adj \to \cL, \enspace \varphi \colon \Sph \to \cL
    \]
    the obvious inclusions (see \autoref{def:wSph}). We make the following observation

    \begin{itemize}
        \item The adjunction $\sph \dashv \iota$ from \cref{thm:sphericalize} restricts to $\sph \colon  \cL \llra \Sph \colon \varphi$ and so we obtain an adjunction
    \[
         \varphi^* \colon \FUNst(\cL,\bA) \llra \FUNst(\Sph, \bA) \colon \sph^*
    \]
    where $\sph^*$ is fully faithful with essential image those functors $G \colon \cL \to \bA$ which invert the maps $\alpha_i \colon \sY(\epsilon) \to \sph\sY (\epsilon)$ for $\epsilon \in \Sigma^{\infty}_2 \Adj$ and $i \in \mathbb{N}$. Recall the endomorphism of $\sY(\epsilon)$ constructed in Section ~\ref{sec:the_sphericalization_construction}. Applying $G$ to this iterated endomorphism and passing to colimits we see that 
    \[
        \colim_{\mathbb{N}}G(\sY(\epsilon)) \simeq G(\sph\sY(\epsilon)).
    \]
    \end{itemize}
    The composite
    \[
        \Sph(\bA) \xrightarrow{s} \Fun^{\on{st}}(\Sigma^{\infty}_2 \Adj,\bA)     \xrightarrow{j_{!}}   \FUNst(\cL, \bA),
    \]
    is fully-faithful and its essential image consists precisely in those cocontinuous functors $G$ for which the maps $G(\alpha_i) \colon G(\sY(\epsilon)) \to G(\sph \sY(\epsilon))\simeq  \colim_{\mathbb{N}}G(\sY(\epsilon))$ are invertible (cf. \autoref{lem:cocompletions}). Our first observation tells us that this is precisely the essential image of $\sph^*$ and thus we obtain an equivalence $\varphi^* \circ j_! \circ s \colon \Sph(\bA) \simeq \FUNst(\Sph, \bA)$. The inverse to this map is given by the factorization of $\pi^*= j^* \circ \sph^*$ to $\Sph(\bA)$. The result thus follows.
\end{proof}

As a consequence of Theorem~\ref{thm:unipropsph} and our construction in Section~\ref{sec:the_sphericalization_construction}, we obtain the following:

\begin{cor}
    For every locally stable 2-category $\bA$, there exists an equivalence of 2-categories,
    \[
        \mathfrak{F} \colon \Sph(\bA) \xlongrightarrow{\simeq} \Sph(\bA),  
    \]
    sending an adjunction $L \colon \cC \llra \cD \colon R$ in $\bA$ to the adjoint pair $T_C^{-1}R \dashv L$. The inverse functor $\mathfrak{F}^{-1}$ assigns to each adjunction as above the adjoint pair $R \dashv T_{\cD}^{-1}L$. We refer to $\mathfrak{F}$ as the \emph{Fourier transform} functor.\footnote{This terminology is originally due to Kapranov and also appears in \cite[Def.2.3]{gammage2022perverse}}
\end{cor}
\begin{proof}
    Let us first sketch the argument before proceeding. First, we will show that for every adjunction $L \dashv R$ we can produce adjunctions $T_{\cC}^{-1}R \dashv L$ and $R \dashv T_{\cD}^{-1}L$. Applying this construction to the generic adjunction in $\Sph$ --and exploiting the universal property of the walking spherical adjunction given in \autoref{thm:unipropsph}-- we will obtain mutually inverse functors $F,F' \colon \Sph \to \Sph$. The desired autoequivalence will therefore be given by restriction along $F$.

    We will only construct the adjunction $T_{\cC}^{-1}R \dashv L$ as the remaining case is formally dual. As before, let $c_{w}\colon T_{\cD} \to LR$ and $t_{w} \colon RL \to  T_{\cC}$ be the maps obtained from the defining square of the twist and cotwist (see Section~\ref{sec:the_sphericalization_construction}). Then we have putative (co)units given by 
    \[
         T_{\cD}^{-1}\ast c_w  \colon \id \to T_{\cD}^{-1}LR \simeq L T_{\cC}^{-1}R, \enspace  T_{\cC}^{-1}\ast t_{w} \colon T_{\cC}^{-1}RL \to \id,
     \]
     where in the first equation we are using \autoref{lem:tmorphadjs}. In order to show that the triangular inequalities hold we consider the diagram,
     \[
\begin{tikzcd}
    P && L && 0 \\
    {T_{\cD}^{-1}L} && {(T_{\cD}^{-1}L)RL} && {T_{\cD}^{-1}L} \\
    {LT_{\cC}^{-1}} && {(LT_{\cC}^{-1})RL} \\
    0 && L
    \arrow[from=1-1, to=1-3]
    \arrow[from=1-1, to=2-1]
    \arrow[from=1-3, to=1-5]
    \arrow["{ T_{\cD}^{-1}\ast c_{w} \ast L}"',from=1-3, to=2-3]
    \arrow[from=1-5, to=2-5]
    \arrow["{T_{\cD}^{-1}\ast L\ast\eta}"', from=2-1, to=2-3]
    \arrow[from=2-1, to=3-1]
    \arrow["{T_{\cD}^{-1}\ast\epsilon\ast L}"', from=2-3, to=2-5]
    \arrow[from=2-3, to=3-3]
    \arrow[from=3-1, to=3-3, swap,"{L\ast T_{\cC}^{-1}\ast \eta}"]
    \arrow[from=3-1, to=4-1]
    \arrow[from=3-3, to=4-3, "{L*T_{\cC}^{-1}*t_{w}}",swap]
    \arrow[from=4-1,to=4-3]
\end{tikzcd}
\]
consisting of bicartesian squares and we claim that $P \simeq 0$. This follows since the second horizontal composite is an isomorphism due to the triangular inequalities for $L\dashv R$. Note that the middle square commutes as the twist and cotwist functors are morphisms of adjunctions (again by \autoref{lem:tmorphadjs}) and thus they intertwine the corresponding (co)units. We conclude that the (central) vertical composite is equivalent to the identity. 

We have produced a new adjunction $T_{\cC}^{-1}R \dashv L$ whose twist and cotwist are given by $T_{\cD}^{-1}$ and $T_{\cC}^{-1}$ respectively. The construction of the adjunction $R \dashv T_{\cD}^{-1}L$ is performed in a similar way and simple unraveling shows that both operations are mutual inverses as desired. 
\end{proof}

\begin{rem}
    A proof of the above result has already been appeared for the $2$-category of stable categories in \cite[Cor.2.5.14]{dyckerhoff2024spherical} using the characterization of spherical functors in terms of semi-orthogonal decompositions. Note that our approach yields a simple and completely formal proof of this fact in any locally stable $2$-category $\bA$, bypassing the theory of semi-orthogonal decompositions.
\end{rem}

Up to this point, we have constructed $\Sph$, a locally stable 2-category that classifies spherical adjunctions. It is possible, however, to make this description more explicit by unpacking the spherical adjunctions $\sph\sY(\epsilon)$ for $\epsilon \in \{ +,- \}$.  We will outline this here, leaving the corresponding proofs to the next section for the interested reader. 

In \autoref{cor:adj} and \autoref{thm:DK-}, we will show that the Dold-Kan equivalence can be upgraded to yield an equivalence of adjunctions;  

\begin{equation}\label{eq:adjunctions}
 \sY(+) \coloneq \left( \begin{tikzcd}
      \fSp_{\geq 0} \arrow[r, bend left =30, "\gr_{\geq 0}"] & \gSp_{\geq 0} \arrow[l,bend left=30,"\triv_{\geq 0}"]
 \end{tikzcd} \right) \quad \quad \quad \quad
 \sY(-) \coloneq \left( \begin{tikzcd}
    \gSp_{\leq 0} \arrow[r, bend left =30, "\triv_{\leq 0}{[}-1{]}\langle 1 \rangle"] & \fSp_{\leq 0} \arrow[l,bend left=30,"\gr_{\leq 0}"]
 \end{tikzcd} \right)
\end{equation}

where the functor $\gr:\fSp\rightarrow \gSp$ associates to a filtered spectrum $(X_{\bullet},f_{\bullet})$ its associated graded $\cof(f_{\bullet -1})\in \gSp$ and the functor $\triv$ associates to a graded spectrum the associated filtered spectrum with $0$-maps.

We can use this description to explicitly compute the sphericalization of $\sY(+)$ and $\sY(-)$:

\begin{thm}\label{thm:sphadj}
    There exist equivalences of adjunctions 
    \begin{equation}
 \fS(\sY(+))\simeq  \left( \begin{tikzcd}
      \fSp \arrow[r, bend left =30, "\gr"] & \gSp \arrow[l,bend left=30,"\triv"]
 \end{tikzcd} \right) \quad \quad \quad \quad
 \fS(\sY(-)) \simeq \left( \begin{tikzcd}
    \gSp \arrow[r, bend left =30, "\triv{[}-1{]}\langle 1 \rangle"] & \fSp \arrow[l,bend left=30,"\gr"]
 \end{tikzcd} \right)
\end{equation}
\end{thm}

To prove this theorem, we first need to compute the twist and cotwist functors of the adjunctions in Equation~\eqref{eq:adjunctions}. As this is a standard computation, we just record the results in the following:

\begin{lem}\label{lem:graded triv}
The twist functor and cotwist functors 
\begin{enumerate}
    \item of the adjunction $\gr_{\geq 0}\vdash \triv_{\geq 0}$ are given by $\sT_{\fSp_{\geq 0}}\simeq \langle-1\rangle[1]$ and $\sT_{\gSp_{\geq 0}}\simeq \langle-1\rangle[1]$ 
    \item of the adjunction $\triv_{\leq 0}[-1]\langle1\rangle\vdash \gr_{\leq 0}$ are given by $\sT_{\gSp_{\leq 0}}\simeq \langle1\rangle[-1]$ and $\sT_{\fSp_{\leq 0}}\simeq \langle1\rangle[-1]$.
\end{enumerate}
In particular, the adjunctions $\gr\vdash \triv$ and $\triv[-1]\langle1\rangle\vdash \gr$ are spherical. 
\end{lem}

\begin{proof}[Proof of Theorem~\ref{thm:sphadj}]
    We proof the claim for $\fS(\sY(+))$ as the other one is analogous. Denote by $i:\bZ_{\geq 0}\rightarrow \bZ$ the canonical inclusion and observe that there exists a commutative diagram
    \[
    \begin{tikzcd}
        \fSp_{\geq 0} \arrow[r,"i_{!}"] & \fSp \\
        \gSp_{\geq 0}\arrow[u,"\triv_{\geq 0}"] \arrow[r,"i^{\simeq}_{!}",swap] & \gSp\arrow[u,swap,"\triv"]
    \end{tikzcd}
    \] 
    It is easy to check that this diagram is vertically left adjointable and hence induces a morphism of adjunctions $(\gr_{\geq 0}\vdash \triv_{\geq 0}) \rightarrow (\gr \vdash \triv)$. Since the target is spherical by \autoref{lem:graded triv}, we obtain by Theorem~\ref{thm:sphericalize} a morphism of adjunctions
    \[
    \fS(\gr_{\geq 0}\vdash \triv_{\geq 0})\rightarrow (\gr \vdash \triv)
    \]
    It follows directly from the formula for the sphericalization functor that the induced morphism is an equivalence.
\end{proof}

\section{Dold-Kan equivalences}\label{sec:dold_kan_equivalences}
The goal of this section is to employ Dold-Kan type equivalences to establish the description of the adjunctions $\sY(+)$ and $\sY(-)$ given in Equation~\eqref{eq:adjunctions}. As these constructions follow a common general pattern, we begin by introducing a general notion of Dold-Kan correspondence in Subsection ~\ref{sub:the_axiomatic_setup}. We then define in Subsections ~\ref{sub:dold_kan_correspondences} and ~\ref{sub:equivalence_of_adjunctions} the needed examples and show that the chosen correspondences induce the desired equivalence of adjunctions.

\subsection{The axiomatic setup}\label{sub:the_axiomatic_setup}

\begin{ntt}\label{not:conejoin}
    Let $\cK$ be a category. We denote by $\cK^{\tre}$ the pushout $\cK^{\tre}\coloneq \cK\times [1]\cup_{\cK\times\{1\}}{[0]}$. We will refer to the additional object of $\cK^{\tre}$ as the \emph{cone point}. We observe that we have a canonical fully-faithful functor $\iota \colon \cK \to \cK^{\tre}$. 
\end{ntt}

\begin{defn}
    For a diagram $\cK^{\tre} \to \cC$ in a stable category we define its \emph{total cofiber}, if it exists, as the composite
    \[
       \totcof \colon  \Fun(\cK^{\tre},\cC)\xrightarrow{\phi} \Fun(\cK^{\tre},\cC)^{[1]} \xrightarrow{\cof} \Fun(\cK^{\tre},\cC)\xrightarrow{\ev_*} \cC,
    \]
    where $\phi$ selects the counit of the adjunction $\iota_! \dashv \iota^*$.
\end{defn}

\begin{ntt}
We denote by $\Cat^{\cone}$ the locally full sub-2-category of $\Fun([1],\Cat)$ on the functors $\cK \to \cK^{\tre}$ and whose morphisms are given by pullback squares. Informally speaking $\Cat^{\cone}$ is the sub-2-category of $\Cat$ generated by categories of the form $\cK^{\tre}$ and functors that only map the cone point to the cone point. 
We will often abuse notation and denote an object $(\cK\hookrightarrow \cK^{\tre})\in \Cat^{\cone}$ by $\cK^{\tre}$.
Note that evaluation at $0$ induces a functor, $\Fgt \colon \Cat^{\cone}\rightarrow \Cat$ that forgets the cone point. 
\end{ntt}

\begin{defn}
    Let $\cF \colon \cJ \to \Cat^{\cone}_{/\cK^{\tre}}$ be a functor and let $\cF^{\times}=\Fgt \circ \cF$. We denote the corresponding cocartesian unstraightenings by
    \[
        (p^{\times},q^{\times}) \colon \cX_{\cF^{\times}} \to \cJ \times \cK^{\tre}, \enspace   (p,q) \colon \cX_{\cF} \to \cJ \times \cK^{\tre},
    \]
    together with the canonical morphism of fibrations $\iota_{\cF} \colon \cX_{\cF^{\times}} \to \cX_{\cF}$. We let $s_{\cF} \colon \cJ \to \cX_{\cF}$ denote the section that selects the fibrewise terminal object. 
\end{defn}

\begin{defn}\label{def:doldkanappendix}
    Let $\cF \colon \cJ \to \Cat^{\cone}_{/\cK^{\tre}}$ and $\cC$ be a stable category such that  $ \iota_{\cF}^* \colon \Fun(\cX_{\cF},\cC) \to \Fun(\cX_{\cF^{\times}},\cC)$, admits a left adjoint. We define a functor $\on{DK}(\cF)$ as the composite
    \[
        \Fun(\cK^{\tre},\cC)\xrightarrow{q^*} \Fun(\cX_{\cF},\cC) \xrightarrow{\phi} \Fun(\cX_{\cF},\cC)^{[1]} \xrightarrow{\cof} \Fun(\cX_{\cF},\cC) \xrightarrow{s_{\cF}^*} \Fun(\cJ,\cC),
    \]
    where $\phi$ selects the counit of the adjunction $\iota_{\cF,!} \dashv \iota_{\cF}^*$.
\end{defn}

\begin{prop}\label{prop:equivdoldkans}
    Let $\cF_0,\cF_1 \colon \cJ \to \Cat^{\cone}_{/\cK^{\tre}}$ be functors and let $\alpha \colon \cF_0 \to \cF_1$ be a natural transformation such that $\alpha^{\times} \colon \cF_0^{\times} \to \cF_1^{\times}$ is pointwise cofinal. Suppose moreover, that $\cJ$ is a linearly ordered set. Let $\cC$ be a stable category such that, for each $i=0,1$, the functor $\iota_{\cF_{i}}^* \colon \Fun(\cX_{\cF_i},\cC) \to \Fun(\cX_{\cF_i^{\times}},\cC)$ admits a left adjoint. Then there is an equivalence $\DK(\cF_0)\simeq \DK(\cF_1)$.
\end{prop}
\begin{proof}
    Minor unraveling reveals that it suffices to show $\alpha$ defines a morphism of adjunctions. To achieve this, we will only need to show that the commutative diagram
    \[
        \begin{tikzcd}
            \Fun(\cX_{\cF_1},\cC) \arrow[r,"\iota_{\cF_{1}}^*"] \arrow[d,"\alpha^*"] & \Fun(\cX_{\cF^{\times}_1},\cC) \arrow[d,"{(\alpha^{\times})^{*}}"] \\
             \Fun(\cX_{\cF_0},\cC) \arrow[r,"\iota_{\cF_{0}}^*"] &  \Fun(\cX_{\cF^{\times}_0},\cC),
        \end{tikzcd}
    \]
    is horizontally left adjointable. Further unpacking then shows that to this end, it will be enough to show that for every object $x \in \cX_{\cF_0}$ which is terminal in its corresponding fibre the functor 
    \[
        {\cX_{\cF_{0}^{\times}}}_{/x} \to {\cX_{\cF_1^{\times}}}_{/\alpha(x)},
    \]
    is cofinal. Recall that the fibration ${\cX_{\cF_i}}_{/x} \to \cX_{\cF_i}$ corresponds to the representable functor on the object $x$. The fact that $\cJ$ is a linearly ordered set and that $x$ is terminal in its corresponding fibre shows that 
    \[
        \cX_{\cF_{i}}(a,x)=\begin{cases}
            \emptyset, \enspace \text{ if  } p_i(a) > p_i(x), \\
            \ast, \enspace \text{ else }.
        \end{cases}
    \]

    This implies that we have a pullback diagram,
    \[
        \begin{tikzcd}
           {\cX_{\cF^{\times}_i}}_{/x} \arrow[r] \arrow[d] &    {\cX_{\cF_i}}_{/x}  \arrow[r] \arrow[d]  & \cJ_{\leq px} \arrow[d] \\
           \cX_{\cF^{\times}_i}  \arrow[r] & \cX_{\cF_i}  \arrow[r] & \cJ
        \end{tikzcd}
    \]
    for $i=0,1$. Taking fibres over the terminal object of $\cJ_{\leq px}$, we obtain a commutative diagram
    \[
        \begin{tikzcd}
            \cF_0^{\times}(px) \arrow[r] \arrow[d] & \cF^{\times}_1(px) \arrow[d] \\
             {\cX_{\cF_{0}^{\times}}}_{/x} \arrow[r] & {\cX_{\cF_1^{\times}}}_{/\alpha(x)}.
        \end{tikzcd}
    \]
    where the vertical functors are cofinal since cocartesian fibrations are smooth (cf. \cite[Proposition 4.1.2.15, Remark 4.1.2.10]{LurieHTT}). Since the top horizontal morphism is cofinal by assumption, we conclude that the bottom map is cofinal as well (cf. \cite[Prop 4.1.1.3]{LurieHTT}).
\end{proof}

The following proposition follows readily from the proof of \autoref{prop:equivdoldkans}.

\begin{prop}\label{prop:existenceandformuladk}
    Let $\cJ$ be a linearly ordered set and let $\cF \colon \cJ \to \Cat^{\cone}_{/\cK^{\tre}}$. Given a stable category $\cC$ admitting colimits of shape $\cF^{\times}(j)$ then the following holds:
    \begin{enumerate}
        \item The functor $\iota_{\cF}^* \colon \Fun(\cX_{\cF},\cC) \to \Fun (\cX_{\cF^{\times}},\cC)$ admits a left adjoint.
        \item For every $X \colon \cK^{\tre} \to \cC$, we have that the value of $\on{DK}(\cF)(X)$ at every $j \in \cJ$ is computed as the total cofiber of the functor
        \[
            \cF(j) \to \cX_{\cF} \to \cK^{\tre} \xrightarrow{X} \cC.
        \]
    \end{enumerate}
\end{prop}

For the construction of the morphism of adjunctions, we need a different description of the functor $\DK(\cF)$. This will be a consequence of the following more general abstract Proposition:

\begin{prop}\label{prop:laxlimitmagic}
    Let $R \colon \cA \to \cB$ be a right adjoint between stable categories. Then the functor $\cA \xrightarrow{\epsilon} \cA^{[1]} \xrightarrow{\cof} \cA$, taking the cofiber of the counit $L \dashv R$ is equivalent to the composite
    \[
        \cA \xrightarrow{\ev_{\cA}^{R}} \lim^{\lax}R \xrightarrow{u_!} \cA^{[1]} \xrightarrow{\ev_1}\cA,
    \]
    where $\ev_{\cA}^{R}$ is the right adjoint to the obvious projection to $\cA$ and $u_!$ is the left adjoint to the canonical restriction $u^* \colon \cA^{[1]} \to \lim\limits^{\lax}R$.
\end{prop}
\begin{proof}
    There exists a functor $\xi \colon \lim\limits^{\lax}R \to \cA^{[1]^{2}}$ which sends every $v \colon R(x) \to y$ representing an object in $\lim\limits^{\lax}R$ to the pushout square
    \[
        \begin{tikzcd}
            LR(x) \arrow[r,"\epsilon"] \arrow[d] & x \arrow[d] \\
            y  \arrow[r] & z {.}
        \end{tikzcd}
    \]
    Restricting this map along $\ev_{\cA}^{R}$ amounts to setting $y=0$ in the previous square. It is easy to verify that $u_!$ is precisely the functor obtained by composing $\xi$ with the functor selecting the right-most vertical arrow in the square above. Similarly, the functor $\cof \circ \epsilon$ can be obtained from the functor $\nu \colon \cA \to \cA^{[1]^{2}}$ sending each $x \in \cA$ to the pushout square
   \[
        \begin{tikzcd}
            LR(x) \arrow[d,"\epsilon"] \arrow[r] & 0 \arrow[d] \\
            x  \arrow[r] & \ell {.}
        \end{tikzcd}
    \]
    after restriction to the terminal edge. It follows that $\xi \circ \ev_{\cA}^{R}$ and $\nu$ agree up to an automorphism of the category $[1]^{\times 2}$ which fixes the terminal object. The result follows. 
\end{proof}

\begin{rmk}\label{rmk:alternativeformula}
     Let $\cF \colon \cJ \to \Cat^{\cone}_{/\cK^{\tre}}$ and $\cC$ be a stable category such that  $ \iota_{\cF}^* \colon \Fun(\cX_{\cF},\cC) \to \Fun(\cX_{\cF^{\times}},\cC)$, admits a left adjoint. Let $\cF_{\square}=\colim\limits^{\oplax}\iota_{\cF}$ and consider the functors
     \[
          \cX_{\cF} \xrightarrow{\alpha} \cX_{\cF_{\square}} \xrightarrow{\beta} \cX_{\cF}\times [1],
      \] 
      Then it follows from \cref{prop:laxlimitmagic} that the functor $\on{DK}(\cF)$ can be computed as the composite 
     \[
        \Fun(\cK^{\tre},\cC)\xrightarrow{q^*} \Fun(\cX_{\cF},\cC) \xrightarrow{\alpha_*} \Fun( \cX_{\cF_{\square}},\cC) \xrightarrow{\beta_!} \Fun(\cX,\cC)^{[1]} \xrightarrow{\ev_1} \Fun(\cX,\cC) \xrightarrow{s_{\cF}^*} \Fun(\cJ,\cC),
     \]
     where $\alpha_*,\beta_!$ denote the corresponding right and left Kan extensions respectively.
\end{rmk}

\subsection{Dold-Kan correspondences}\label{sub:dold_kan_correspondences}

\begin{defn}[Augmented Lurie Dold-Kan]\label{defn:LDK}
We denote by $\Delta_{+,\leq k}^{\op}$ the full subcategory of $\Delta_{+}^{\op}$ generated by objects $[n]$ with $n\leq k$. Note that we have that $\Delta_{\leq k}^{\op,\tre} \simeq \Delta_{+,\leq k}^{\op}$ for every $k \geq -1$.

We have a functor $\cF_{+} \colon  \bZ_{\geq 0} \to \Cat_{/\Delta_{+}^{\op}}$ defined as $\cF_{+}(k+1) = \Delta_{\leq k,+}^{\op}$ where for $k+1 < k+2$ we assign the obvious functor induced by $d_{k+1} \colon [k] \to [k+1]$. 
\end{defn}

\begin{ntt}
    Let $X_{\bullet}:\Delta_{+}^{\op}\rightarrow \cC$ be an augmented simplicial object. We denote by $X_{\bullet}^{\leq k}$ the restriction of $X_{\bullet}$ to $\Delta_{+,\leq k}^{\op}$.
\end{ntt}

\begin{prop}[Augmented Lurie Dold-Kan]\label{prop:LDK}
    Let $\cC$ be a stable category. Then we have an equivalence of categories
    \[
    \on{LDK}_{+}\coloneq \DK(\cF_{+}) \colon  \Fun(\Delta_{+}^{\op},\cC) \xrightarrow{\simeq} \Fun(\bZ_{\geq 0},\cC),
    \]
    that maps an augmented simplicial $X_{\bullet}$ to the filtered object
    \[
    \totcof(X^{\leq -1}_{\bullet}) \rightarrow \totcof(X^{\leq 0}_{\bullet}) \rightarrow \totcof(X^{\leq 1}_{\bullet}) \rightarrow \dots
    \]
\end{prop}
\begin{proof}
   The existence of the functor and the corresponding formula follow directly from \autoref{prop:existenceandformuladk}. The fact that $\on{LDK}_{+}$ is an equivalence follows from the same argument given in \cite[Lem.1.6]{gammage2022perverse}. 
\end{proof}

For the construction of the equivalence of adjunctions it is more convenient to use the following variant of the Dold-Kan equivalence:

\begin{defn}\label{defn:tildecF}
        We denote by $\cJ_{n}\subset (\Delta^{\op}_{+,\leq n})_{/{[n]}}$ the full subcategory generated by injective maps and the terminal object. We have a functor $\tilde{\cF}_{+}\colon \bZ_{\geq 0} \rightarrow \Cat_{/\Delta_{+}^{\op}}$ defined as 
    \[
   \tilde{\cF}_{+}(k+1)=
    \cJ_{k} 
    \]
    where for $k+1<k+2$ we assign the obvious functor induced by $d_{k+2}$. We also note that we have a natural transformation $\tilde{\cF}_{+}\rightarrow \cF_{+}$
\end{defn}

\begin{lem}
    Let $\cC$ be a stable category. The functor $\LDK_{+}$ is equivalent to the functor $\DK(\Tilde{\cF}_{+})$.
\end{lem}
\begin{proof} 
Combine \autoref{prop:equivdoldkans} together with \cite[Lem.1.2.4.17]{LurieHA}
\end{proof}

In the following, we will also need the following Dold-Kan correspondence

\begin{defn}\label{defn:TDKcons}
    For every $n \geq 0$ we denote by $\cJ_{n}^{\max}\subset \cJ_{n}$ the sub-category on the maximum preserving maps $[k]\hookrightarrow [n]$ with $k\geq 0$. Note that $\cJ_{n}^{\max}$ admits a final object.

    We have a functor $\cF_{\infty} \colon \bZ^{\simeq}_{\geq 0} \rightarrow \Cat_{/\Delta_{\infty}}$ defined as
    \[
    \cF_{\infty}(k) = 
    \cJ_{k}^{\max}.
    \]
    We denote the induced Dold-Kan correspondence by $\DK_{\infty}$.
\end{defn}

\subsection{Equivalence of Adjunctions}\label{sub:equivalence_of_adjunctions}

The goal of this section is to show that the Dold-Kan correspondences constructed in the last section combine to an equivalence of adjunctions. For this we need the following construction

\begin{const}
    Let $\tilde{\cF}_+$ be as in \autoref{defn:tildecF} and let $\tilde{\cF}_{[2]}$ be the functor obtained as $\tilde{\cF}_+(i) \times [2]= \tilde{\cF}_{[2]}(i)$. We define a subfunctor of $\tilde{\cF}_{\diamond} \subset \tilde{\cF}^{\times}_{[2]}$ by selecting for every $n+1 \in \bZ_{\geq 0}$ the  subposet $\cJ_{n}^{\diamond}\subset \cJ_{n} \times [2]$ containing:
    \begin{itemize}
        \item  Every object belonging to $\cJ_{n} \times \{0<2\}$ except the terminal object.
        \item Every object object of $(\cJ_{n}^{\max} \times \{i\} )\times \{0<1\}$ except the corresponding terminal object for $i=0,1$, where we have used the canonical equivalence $\cJ_{n}\simeq \cJ_{n}^{\max}\times [1]$
    \end{itemize}
    We conclude that we have natural transformations between the unstraightened functors 
    \[
        \cX_{\tilde{\cF}} \xrightarrow{\mu} \cX_{\tilde{\cF}_{\diamond}} \xrightarrow{\omega} \cX_{\tilde{\cF}_{[2]}} \simeq \cX_{\tilde{\cF}_{+}} \times [2],
    \]
    which are fibrewise fully-faithful.
\end{const}

\begin{prop}\label{prop:interpolatingLDK}
    The functor $\on{LDK}_+$ can be equivalently expressed as the composite
    \[
       \Gamma \colon   \Fun(\Delta_+^{\op},\cC) \xrightarrow{q^*} \Fun(\cX_{\tilde{\cF}_{+}}, \cC)\xrightarrow{\mu_*}\Fun(\cX_{\tilde{\cF}_{\diamond}},\cC) \xrightarrow{\omega_!}\Fun ( \cX_{\tilde{\cF}_{[2]}},\cC) \xrightarrow{s_{{\cF_{[2]}}}} \Fun(\bZ_{\geq 0},\cC).
    \]
\end{prop}
\begin{proof}
    We commence the proof by observing that $\mu_*$ is defined since the corresponding slices indexing the limits in the point-wise formula either have an initial element or are empty. To verify the existence of $\omega_!$ we pick an element $x \in \cX_{\tilde{\cF}_{[2]}}$ which is not in the image of $\omega$. If $x$ is terminal in the corresponding fibre the proof proceeds totally analogously as in the proof of \autoref{prop:equivdoldkans}. We will deal now with two separate cases:
    \begin{itemize}
        \item The object $x$ corresponds to the terminal object in $(\cJ_{n}^{\max} \times \{0\} )\times \{0<1\}$. Let $w \in \cX_{\tilde{\cF}_{+}}$ be the terminal object in the fibre over $n$. Then it follows that we have an identification ${\cX_{\tilde{\cF}_{+}^{\times}}}_{/w} \simeq {\cX_{\tilde{\cF}_{\diamond}}}_{/x}$ and so our argument follows from \autoref{prop:existenceandformuladk}.
        \item The object $x$ corresponds to the terminal object in $(\cJ_{n}^{\max} \times \{1\} )\times \{0<1\}$. In this case we can produce an auxiliary fibration $\cX_{\tilde{\cF}_{\downarrow}}$ obtained from ${\cX_{\tilde{\cF}_{\diamond}}}$ by adding the terminal objects of $(\cJ_{n}^{\max} \times \{0\} )\times \{0<1\}$. We obtain a factorization of the form $\omega=u \circ v$ and the previous bullet point implies that $v_!$ exists. The existence of $u_!$ follows from precisely the same argument as in the previous case.
    \end{itemize}
    To finish the proof we will show that $\Gamma \simeq \on{LDK}_+$. For this we will make use of the description of the former functor given in \autoref{rmk:alternativeformula}. Borrowing the notation from loc. cit, it follows that the functor $\{0<2\} \to [2]$ induces a commutative diagram of fibrations
   \[\begin{tikzcd}
    & {\cX_{\tilde{\cF}_{\square}}} & {\cX_{\tilde{\cF}_{+}}\times [1]} \\
    {\cX_{\tilde{\cF}_{+}}} & {\cX_{\tilde{\cF}_{\diamond}}} & {\cX_{\tilde{\cF}_{+}}\times[2]}
    \arrow["\beta", from=1-2, to=1-3]
    \arrow["i", from=1-2, to=2-2]
    \arrow["j", from=1-3, to=2-3]
    \arrow["\alpha", from=2-1, to=1-2]
    \arrow["\mu"', from=2-1, to=2-2]
    \arrow["\omega"', from=2-2, to=2-3]
\end{tikzcd}\]
   It is easy to verify that $i^* \circ \mu_* \simeq \alpha_*$ so we will focus our attention into showing that the diagram
   \[
       \begin{tikzcd}
           \Fun({\cX_{\tilde{\cF}_{+}}\times[2]}, \cC) \arrow[d,"j^*"] \arrow[r,"\omega^*"] &  \Fun({\cX_{\tilde{\cF}_{\diamond}}}, \cC) \arrow[d,"i^*"] \\
           \Fun({\cX_{\tilde{\cF}_{+}}\times [1]},\cC) \arrow[r,"\beta^*"] & \Fun({\cX_{\tilde{\cF}_{\square}}},\cC)
       \end{tikzcd}
   \]
   is horizontally left adjointable. Let $(\cJ_n \times [1])^{\times}$ denote the poset obtained from the cartesian product by removing the terminal object. Unraveling the definitions, we see that it will enough to show that the map $\varphi \colon (\cJ_n \times [1])^{\times} \to \cJ_{n}^{\diamond}$ is cofinal. Fortunately, the map $\varphi$ admits a left adjoint induced by the map $s_1 \colon [2] \to [1]$ and so $\varphi$ is cofinal. The result now follows easily.
\end{proof}

\begin{const}\label{const:laxcommuting}
    Let $x_{n} \in \cJ_{n}^{\max}$ denote the maximal element. There exists a functor $\beth \colon \bZ_{\geq 0}^{\simeq }\times [1]^{2} \to \cX_{\tilde{\cF}_{[2]}}$ selecting the diagram $(\{x_n+1\}\times [1]) \times \{1<2\}$. After a minor unraveling the proof of \autoref{prop:interpolatingLDK}, we can view the composite
    \[
        \Fun(\Delta_+^{\op},\cC) \xrightarrow{q^*} \Fun(\cX_{\tilde{\cF}_{+}}, \cC)\xrightarrow{\mu_*}\Fun(\cX_{\tilde{\cF}_{\diamond}},\cC) \xrightarrow{\omega_!}\Fun ( \cX_{\tilde{\cF}_{[2]}},\cC) \xrightarrow{\beth^*} \Fun(\bZ_{\geq 0}^{\simeq},\cC)^{[1]^{2}},
    \]
    as a diagram $D \colon [1]\times [1] \to \Fun( \Fun(\Delta_+^{\op},\cC),\Fun(\bZ_{\geq 0}^{\simeq},\cC))$ of the form
    \[
        \begin{tikzcd}
            \on{DK}_{\infty}\iota^* \arrow[d] \arrow[r] & (\on{LDK}_{+})_{|\bZ_{\geq 0}^{\simeq}} \arrow[d] \\
            0 \arrow[r] & (\on{LDK}_{+}\langle 1 \rangle)_{|\bZ_{\geq 0}^{\simeq}}
        \end{tikzcd}
    \]
    where we are denoting $\iota^* \colon \Fun(\Delta_+^{\op},\cC) \to \Fun(\Delta_{\infty}^{\op},\cC)$ and the right-most vertical map is induced by the natural transformation $\Fun(\bZ_{\geq 0},\cC) \times [1] \to \Fun(\bZ_{\geq 0}, \cC)$ between the identity and the functor $\langle 1 \rangle$ that shifts the filtration by $1$.
\end{const}

\begin{prop}\label{prop:TDKLDK}
    There exists a commutative diagram of stable categories
    \[
        \begin{tikzcd}
            \Fun(\Delta_{+}^{\op},\cC) \arrow[d,"\iota^*"] \arrow[r,"\on{LDK}"] &  \Fun(\bZ_{\geq 0},\cC) \arrow[d,"\langle 1\rangle\gr_{\geq 0}{[-1]}"] \\
            \Fun(\Delta_{\infty}^{\op},\cC) \arrow[r,"\on{DK}_{\infty}"] & \Fun(\bZ^{\simeq}_{\geq 0},\cC)
        \end{tikzcd}
    \]
    which is vertically left adjointable and where the right-most vertical functor is given by the associated graded functor composed with the loops and shift by $1$-functor.
\end{prop}
\begin{proof}
\autoref{const:laxcommuting} provides us the the existence of a laxly commutative square. In order to verify that the square is commutative we need to show that the square we have selected is actually bicartesian. This follows easily from our construction together the results of \cite[Appendix A]{dyckerhoff2019simplicial}.

It remains to show that the diagram is vertically right adjointable. A minor unraveling shows that the mate of $\eta$ evaluated on $Y_{\bullet}:\Delta_{\infty}^{\op}\rightarrow \cC$ is given by the composite
\[
\triv\circ\DK_{\infty}(Y_{\bullet})(n)\simeq \totcof(Y_{\bullet}\vert_{\cJ^{\max}_{n}}) \rightarrow \totcof(Y_{\bullet+1}\vert_{\cJ^{\max}_{n}}) \rightarrow 
\totcof(Y_{\bullet}\vert_{\cJ^{\max}_{n}}) \simeq \LDK_{+}\circ t^{\ast}(Y_{\bullet})(n)
\]
where the maps are induced by the commutative diagram
\[
\begin{tikzcd}
    Y_{\bullet}\arrow[r,swap,"s_{\bullet +1}"] \arrow[rr,bend left= 30,"\id"]&  Y_{\bullet+1} \arrow[r,swap,"\partial_{\bullet+1}"] & Y_{\bullet}
\end{tikzcd}
\]
In particular, the composite is an equivalence and the diagram is vertically adjointable.
\end{proof}

\begin{thm}\label{thm:DK+}
    The commutative diagram in \autoref{prop:TDKLDK} defines an equivalence in $\Adj(\St)$.
\end{thm}
\begin{proof}
    We showed in in \autoref{prop:TDKLDK} that the diagram is adjointable which precisely says that it defines a morphism of adjunctions. Invoking \autoref{prop:LDK} we are left to show that $\DK_{\max}$ is also an equivalence. Note that also by \autoref{prop:LDK} and the fact that our diagram is a morphism of adjunctions it will suffice to show that both adjunctions are comonadic. To this end we will verify the conditions of the monadicity theorem found in \cite[Thm.4.7.3.5]{LurieHA}.

     We observe that since $\on{t}\colon \Delta_+ \to \Delta_{\infty}$ is essentially surjective it follows that $t^{*}$ is conservative and similarly for the left adjoint $\triv$. It is also clear that both functors are (co)limit preserving. The conditions in \cite{LurieHA} thus hold provided $\cC$ is sufficiently complete. We consider the commutative diagram
     \[
         \begin{tikzcd}
             \Fun(\Delta_\infty^{\op},\cC) \arrow[r] \arrow[d] & \Fun(\bZ^{\simeq}_{\geq 0},\cC) \arrow[d] \\
              \Fun(\Delta_\infty^{\op},\on{Ind}(\cC)) \arrow[r,"\simeq"] &  \Fun(\bZ^{\simeq}_{\geq 0},\on{Ind}(\cC))
         \end{tikzcd}
     \]
     where the vertical functors and hence the horizontal functors are fully-faithful. To finish the proof we will need to show that bottom functor has the following property:
     \begin{itemize}
         \item Let $F \colon \Delta_{\infty}^{\op} \to \Ind(\cC)$ such that $\DK_{\max}(F)$ takes values in $\cC$. Then so does $F$. 
     \end{itemize}
     We will verify that $F(k) \in \cC$ inductively on $k \in \Delta_\infty$. The case $k=0$ is clear. For the remaining cases it follows that $F(k)$ is the totalization of a cube where all objects except the initial vertex belong to $\cC$ by the induction hypothesis. Since the functor $\cC \to \Ind(\cC)$ preserves limits and cocartesian cubes are also cartesian in $\cC$ the claim follows.
\end{proof}

\begin{cor}\label{cor:adj}
    The commutative diagram in \autoref{prop:TDKLDK} is $2$-fold vertically left adjointable. In particular, it defines an equivalence of adjunctions between $(\gr_{\geq 0} \vdash \triv_{\geq 0})$ and $\sY(+)$.
\end{cor}
\begin{proof}
    It follows from Proposition~\ref{prop:TDKLDK} that the square is vertically right adjointable. Since the upper horizontal morphisms are equivalences, it follows that the resulting diagram is itself horizontably right adjointable. As a result, we obtain a square consisting of right adjoints. It follows that the invertible $2$-morphism obtained from this square by passing to adjoints is the required mate. 
\end{proof}

\begin{thm}\label{thm:DK-}
    There exists a commutative diagram of stable categories
    \[
    \begin{tikzcd}
        \Fun(\Delta_{+},\cC) \arrow[r,"\simeq"] \arrow[d,swap,"\iota^{\ast}"] & \Fun(\bZ_{\leq 0},\cC) \arrow[d,"\gr_{\leq 0}"] \\
        \Fun(\Delta_{\infty},\cC) \arrow[r,"\simeq"] & \Fun(\bZ_{\leq 0}^{\simeq},\cC)
    \end{tikzcd}
    \]
    that is vertically right and left adjointable. In particular, we obtain an equivalence of adjunctions of $\triv_{\leq 0}[-1]\{1\}\vdash \gr_{\leq 0}$ with $\sY(-)$
\end{thm}
\begin{proof}
    Passing to opposite categories in the diagram of Proposition~\ref{prop:TDKLDK} induces a commutative square
\[
\begin{tikzcd}
    \Fun(\Delta_{+},\cC)\arrow[r,"\simeq"]\arrow[d,"\iota^{\ast}"] & \Fun(\Delta_{+}^{\op},\cC^{\op})^{\op} \arrow[d,"((\iota^{\op}))^{\ast})^{\op}"]\arrow[r,"\LDK^{\op}"] & \Fun(\bZ_{\geq 0},\cC^{\op})^{\op}\arrow[d,"(\langle 1\rangle\gr_{\geq 0}{[-1]})^{\op}"] \arrow[r,"\simeq"] &  \Fun(\bZ_{\leq 0},\cC) \arrow[d,"\gr_{\leq 0}"] \\
\Fun(\Delta_{\infty},\cC) \arrow[r,"\simeq"] & \Fun(\Delta_{\infty}^{\op},\cC^{\op})^{\op} \arrow[r,"\DK_{\infty}"] & \Fun(\bZ_{\geq 0}^{\simeq},\cC^{\op})^{\op}\arrow[r,"\simeq"] & \Fun(\bZ_{\leq 0}^{\simeq},\cC)
\end{tikzcd}
\]        
 that is vertically right adjointable. It follows as in the proof of Corollary~\ref{cor:adj} that it is also vertically left adjointable.

\end{proof}

\bibliographystyle{halpha}
\bibliography{sphreferences.bib}

\end{document}